\documentclass[reqno,10pt,a4paper]{amsart}
\usepackage[utf8]{inputenc}
\usepackage[T1]{fontenc}
\usepackage[english]{babel}
\usepackage{amssymb}
\usepackage[foot]{amsaddr}
\usepackage{a4wide}

\usepackage{graphicx}
\usepackage{wrapfig}
\usepackage{enumerate}
\usepackage[usenames,dvipsnames]{color}

\usepackage{mathtools}
\mathtoolsset{showonlyrefs,showmanualtags}

\usepackage{esint}

\usepackage{pstricks,pst-plot,pst-math} 
\usepackage{pstricks-add}

\allowdisplaybreaks

\newtheorem{thm}{Theorem}[section]

\newtheorem*{thmxx}{Main Theorem}
\newtheorem*{thmxx*}{Main Theorem*}
\newtheorem*{thmxx**}{Main Theorem**}

\newtheorem{defn}[thm]{Definition}
\newtheorem{lem}[thm]{Lemma}
\newtheorem{prop}[thm]{Proposition}

\newtheorem*{conjecture*}{Conjecture}

\newtheorem{innercustomgeneric}{\customgenericname}
\providecommand{\customgenericname}{}
\newcommand{\newcustomtheorem}[2]{%
	\newenvironment{#1}[1]
	{%
		\renewcommand\customgenericname{#2}%
		\renewcommand\theinnercustomgeneric{##1}%
		\innercustomgeneric
	}
	{\endinnercustomgeneric}
}
\newcustomtheorem{customthm}{Theorem}

\theoremstyle{definition}

\theoremstyle{remark}
\newtheorem{rem}[thm]{Remark}

\numberwithin{equation}{section}

\newcommand{\DeclareAutoPairedDelimiter}[3]{%
	\expandafter\DeclarePairedDelimiter\csname Auto\string#1\endcsname{#2}{#3}%
	\begingroup\edef\x{\endgroup
		\noexpand\DeclareRobustCommand{\noexpand#1}{%
			\expandafter\noexpand\csname Auto\string#1\endcsname*}}%
	\x}

\DeclareAutoPairedDelimiter{\abs}{\lvert}{\rvert}
\DeclareAutoPairedDelimiter{\norm}{\lVert}{\rVert}
\DeclareAutoPairedDelimiter{\bra}{(}{ )}
\DeclareAutoPairedDelimiter{\pra}{[}{]}
\DeclareAutoPairedDelimiter{\set}{\{}{\}}
\DeclareAutoPairedDelimiter{\skp}{\langle}{\rangle}

\DeclareMathAlphabet{\mathup}{OT1}{\familydefault}{m}{n}
\newcommand{\dx}[1]{\mathop{}\!\mathup{d} #1}

\DeclareMathOperator{\loc}{loc}

\DeclareMathOperator{\tr}{tr}

\DeclareMathOperator{\Lip}{Lip}

\newcommand{\N}{\mathbb{N}}
\newcommand{\R}{\mathbb{R}}

\newcommand{\cC}{\ensuremath{\mathcal C}}

\newcommand{\cN}{\ensuremath{\mathcal N}}

\newcommand{\coeff}{\ensuremath{c}}

\usepackage[colorinlistoftodos,prependcaption,textsize=tiny,linecolor=red,backgroundcolor=red!25,bordercolor=red]{todonotes}

\usepackage[pdfusetitle]{hyperref}
\hypersetup{
	colorlinks=true,       
	linkcolor=black,          
	citecolor=black,        
	filecolor=black,      
	urlcolor=black           
}

\title[The regular part of the free boundary close to singularities]{The structure of the regular part of the free boundary close to singularities in the obstacle problem}

\author{Simon Eberle$^1$}
\address{$^1$Basque Center for Applied Mathematics, Spain}
\email{seberle@bcamath.org}
\author{Henrik Shahgholian$^2$}
\address{$^2$Department of Mathematics, KTH Royal Institute of Technology, Sweden}
\email{henriksh@kth.se}
\author{Georg S. Weiss$^3$}
\address{$^3$Faculty of Mathematics, University of Duisburg-Essen, Germany}
\email{georg.weiss@uni-due.de}

\let\rho\varrho
\let\phi\varphi
\let\epsilon\varepsilon
\let\emptyset\varnothing

\thanks{H. Shahgholian was supported in part by the Swedish Research Council. {S. Eberle was supported by the European Research Council (ERC) under the Grant Agreement No 948029}}

\begin{document}
	
	\begin{abstract}
We prove the –to the best knowledge of the authors– first result on the fine asymptotic behavior of the regular part of the free boundary of the obstacle problem close to singularities. The result is motivated by a recent answer to a long standing conjecture concerning the classification of global solutions of the obstacle problem.
	\end{abstract}

	\maketitle

	\tableofcontents

\section{Introduction}

Many important problems in science, finance and engineering can be modeled by PDEs that have a-priori unknown interfaces. Such problems are called \emph{free boundary problems} and they have been a major area of research in the field of PDEs for at least half a century.
The \emph{obstacle problem} arguably is the most extensively   studied free boundary problem. It may be derived from a simple model for spanning an elastic membrane over some given (concave) obstacle. Alternatively it can be derived from a certain setting in the Stefan problem, the simplest model for the melting of ice, or from the Hele-Shaw problem. 
The obstacle problem may be expressed in the scalar nonlinear partial differential equation 
\begin{align}\label{eq:obstacle-local}
\Delta u=\coeff(x)\chi_{\{u>0\}}, ~ u \geq 0 , ~ \text{ where } \coeff \in \Lip(\R^N), ~c \geq \coeff_0 >0 \text{ in } \R^N,
\end{align}
and  the set $\{ u>0\}$ models 
the phase where the membrane does not touch the obstacle
respectively  the water phase in the stationary Stefan problem. 
The set $\{ u=0\}$ is called {\em coincidence set} and the interface $\partial\{ u>0\}$ is called the {\em free boundary}. \\
From the point of view of mathematics, the most challenging question in free boundary problems is the \emph{structure and the regularity of the free boundary}. The development of contemporary regularity theory for free boundaries started with the seminal work of L. Caffarelli \cite{Caffarelli_regularity_free_boundary_higher_dimensions_Acta77} in the late 70s, and since then has been a very active field of research. \\ As the obstacle problem has been extensively studied over the last $4$ decades  of the questions originally posed only the hard problems have remained unsolved. Among them are the fine structure of the singular part of the free boundary and the behavior of the regular part of the free boundary close to singularities. Towards answering the first question there have been many impressive results in recent years (cf. A. Figalli, J. Serra \cite{Figalli-Serra-2020-Inventiones}, A. Figalli, J. Serra and X. Ros-Oton \cite{Figalli-Ros-Oton-Serra-Generic}, O. Savin and H. Yu \cite{SavinYu2020regularity}). While the singular set has thus been extensively studied there has -- to the authors' best knowledge -- been  no result on the behavior of the \emph{regular part of the free boundary close to singularities}.

\subsection{The free boundary in the obstacle problem close to singularities}

From \cite{Schaeffer_examples_of_singularities77} we know that the behavior may for $C^\infty$-coefficients
$\coeff$ be quite complicated, including the possibility of infinitely many cusp domains
accumulating at a singular point.
In the present paper the authors give a new description of the 
{\em regular set close to singularities}.
Let us state here the precise result of the present article.
\begin{figure}[h!]
	\centering	
	\frame{
		\psset{xunit=1cm,yunit=1cm}
		\begin{pspicture}		
		(-5,-3)(8,3)	
		\rput[0,1.5](0,0){{$x^0$}}
		\pscurve
		(-5,-1.5)(-4,-0.9)(-3,-0.5)(-1.5,-0.19)(0,0)
		\pscurve
		(0,0)(-1.5,0.19)(-3,0.5)(-4,0.9)(-5,1.5)
		\psellipse(1.2,0)(0.5,0.08)
		\psellipse(3,0)(0.6,0.18)
		\pscurve
		(4.5,0)(4.8,-0.04)(5.4,-0.24)(5.99,-0.3)(6.3,0)(5.99,0.3)(5.4,0.24)(4.8,0.04)(4.5,0)
		\psdot(0,0.0)
		
		\psdot(4.5,0)
		\rput[-12.5,1.5](0,0){{$x^3$}}
		
		\psdot(3,0.18)
		\rput[-8,-1.3](0,0){{$x^2$}}
		
		\psdot(2.42,0)
		\rput[-6.2,1.7](0,0){{$x^1$}}
		
		\end{pspicture}
	}  
	\caption{Local structure of the free boundary close to singularities whose blow-down limit has a one dimensional coincidence set.}
	\label{fig:local_structure_of_coincidence_sets}
\end{figure}
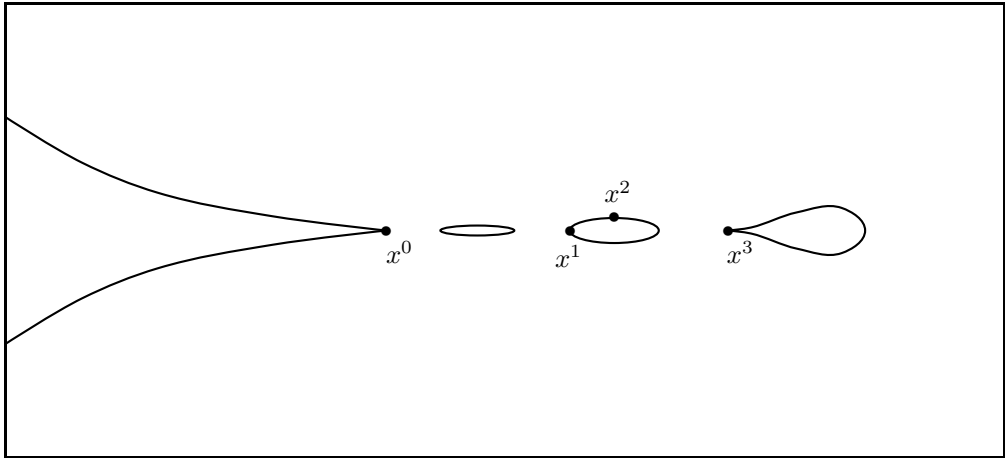

\begin{thmxx}\label{thm:MainTheoremII}
	\mbox{} \\
	Let $u$ be a solution of the obstacle problem \eqref{eq:obstacle-local} in an open set $\Omega \subset \R^N$, and let $x^0 \in \partial \set {u>0} \cap \Omega$ be an order $n$ singular point of the free boundary (cf. Definition \ref{solandcod} \eqref{item:definition_order_1_singularity}) where $n \geq 1$, such that
	\begin{align}
	\frac{u(rx+x^0)}{r^2} \to p(x) \quad \text{ in } C_{\operatorname{loc}}^{1,\alpha}(\R^N) \text{ as } r \to 0,
	\end{align}  
	where
	\begin{align} 
	p(x) \geq c_{p} \abs{x}^2 \quad \text{ for some } c_{p}>0 \text{ and for all } x \in \R^{N-n} \times \set {0}^n.
	\end{align} 
	Let $E'\subset \R^{N-n}$ be the unique ellipsoid of diameter $1$ in Proposition \ref{thm:MainTheorem_Intro_I} with respect to the polynomial $p'(x') := p(x',0)$ for all $x' \in \R^{N-n}$.
	Then there exists $\delta >0$ 
	such that for
	any free boundary point $x \in \partial \set {u>0} \cap B_\delta(x^0)$, setting for $\R^N \ni x = (x',x'') \in \R^{N-n} \times \R^n$
	\begin{align}
	& d(x'') := \operatorname{diam}\bra { \set {u(\cdot,x'')=0} \cap B'_{2\delta}((x^0)')},
	\end{align}
	it holds that:
	\begin{enumerate}[(i)]
		\item \label{item:first_part_of_theorem_singularities}
		if $d(x'')>0$ there is $t': \R^n \to \R^{N-n}$, $t'(x'') \to 0$ as $x'' \to (x^0)''$ and $\omega(s)\to 0$ as $s \to 0$ such that
		either 
		\begin{enumerate}
			\item \label{item:first_part_of_theorem_singularities_a}
			$\{ y\in B_{2d(x'')}(x): u(y)=0\}$ is (in $C^2$) $\omega(|x-x^0|)d(x'')$-close to
			\begin{align}
			(t'(x'') + {d(x'')}E') \times B''_{2d(x'')}
			\end{align}
			or
			\item  \label{item:first_part_of_theorem_singularities_b}
			$\{ y \in B_{2 d(x'')}(x) :~ u(y) =0 ~,~ (x-y)'' \cdot \nu''(x) =0 \} $ 
			is in the hyperplane $\{y \in \R^N :  (x-y)'' \cdot \nu''(x) =0  \}$  (in $C^2$) $\omega(|x-x^0|)  d(x'')$-close to 
			\begin{align}
			t'(x'') + { \set{ y \in B_{2d(x'')}(x): y' \in \sqrt{ \max \{1+c( y-x)'' \cdot \nu''(x), 0 \}   } ~ d(x'')  E'    }    },
			\end{align}
			where {$c>0$,} $\nu'': \partial \set {u=0 } \to \partial B_1'' \subset \R^n$,
			\begin{align} \label{eq:def_of_nu2}
			&\nu''(x) :=  \frac{  \int_{ \{u=0\} \cap B_{d(x'')}(x)  } (x-y)'' \dx{y} }{ \abs{\int_{ \{u=0\} \cap B_{d(x'')}(x)  } (x-y)'' \dx{y}}   }
			\end{align}
			and
			\begin{align}
			&\operatorname*{osc} \limits_{ y \in B_{d(x'')}(x) \cap \{u=0 \}} \nu''(y) \to 0 ~ \text{ as } x \to x^0. 	
			\end{align}
		\end{enumerate}
		\item  \label{item:second_part_of_theorem_singularities}
		if $d( x''  )=0$ then setting 
		\begin{align} I_\delta := \{  y'' \in \R^n : \set {u(\cdot , y'') =0} \cap B'_{2\delta}((x^0)') \neq \emptyset   \},
		\end{align}
		either 
		\begin{enumerate}
			\item \label{item:second_part_of_theorem_singularities_a}
			$x$ is a singular free boundary point
			and
			\begin{align}
			\lim_{ \substack{y'' \in I_\delta \\  y''  \to x'' }  } \frac{d( y'' )-d( x''  )}{  \abs{y''-x''}    } =0,
			\end{align}
			or
			\item \label{item:second_part_of_theorem_singularities_b}
			$x$ is a regular free boundary point and
			\begin{align}
			\lim_{ \substack{y'' \in I_\delta \\  y'' \cdot \nu''(x) \to x'' \cdot \nu''(x'') }  } \frac{d( y'' )-d( x''  )}{\sqrt{| y'' \cdot \nu''(x) - x'' \cdot \nu''(x)|}} \in \R,
			\end{align}
			where $\nu''$ is as in \eqref{eq:def_of_nu2} and 
			\begin{align}
			&\operatorname*{osc} \limits_{ y \in B_{d(x'')}(x) \cap \{u=0 \}} \nu''(y) \to 0 ~ \text{ as } x \to x^0. 	
			\end{align}
		\end{enumerate}
	\end{enumerate}
\end{thmxx}

\begin{rem}[$d(x'')$ in the Main Theorem]
Note that since $x^0$ is an order-$n$ singular point for every $\epsilon >0$ there is $\delta >0$ such that $d(x'')< \epsilon$ for all $x \in B_\delta(x^0)$.
\end{rem}
\begin{rem}[The Main Theorem in a picture]
In Figure \ref{fig:local_structure_of_coincidence_sets} the case \eqref{item:first_part_of_theorem_singularities} (\eqref{item:first_part_of_theorem_singularities_a} and \eqref{item:second_part_of_theorem_singularities_b}) of the Main Theorem is represented by $x^2$, the case \eqref{item:second_part_of_theorem_singularities_a} is represented by $x^3$ and \eqref{item:second_part_of_theorem_singularities_b} is depicted in $x^1$.
\end{rem}

\begin{rem}
By a $C^2$-surface we mean a set that can locally be represented by a $C^2$-graph. By `close in $C^2$' we mean that locally the $C^2$-surfaces are -- as $C^2$-graphs -- close in $C^2$-norm.
\end{rem}

We conjecture that for $C^\infty$-coefficients $\coeff$ paraboloids do occur as coincidence sets of blow-up limits with moving centers approaching singularities.

\section{Notation and known results}\label{sec:notation}

We shall now  clarify the notation used in the introduction. 
\\
Throughout this work $\R^N$ will be equipped with the Euclidean inner product $x \cdot y$ and the induced norm $\abs{x}$. Due to the nature of the problem we will often write $x \in \R^N$ as $x=(x', x'') \in \R^{N-n} \times \R^n$ for an $n \in \N$, $N \geq n+1$. $B_r(x)$ will be the open $N$-dimensional ball of center $x$ and radius $r$. $B_r'(x')$ will be the open $(N-n)$- dimensional ball of center $x' \in \R^{N-n}$ and radius $r$ and $B''_r(x'')$ will be the open $n$-dimensional ball of center $x'' \in \R^n$ and radius $r$. Whenever the center is omitted it is assumed to be $0$.
\\
When considering a set $A$, $\chi_A$ shall denote the characteristic function of $A$.
\\
Finally we call special solutions of the form
$\max(x\cdot e,0)^2/2$ half-space solutions; here $e\in \partial B_1$ is a fixed vector.\\
When $M \in \R^{N \times N}$ is a matrix, we mean by $\tr(M):= \sum_{j=1}^N M_{jj}$ its trace. 
\\ The following proposition is a consequence of \cite{EFW_Complete_classification} (or for some cases \cite{EberleShahgholianWeiss_global_solutions_and_local_analysis_2020}).

\begin{prop}\label{thm:MainTheorem_Intro_I}
	Let  $u$ be a   solution of  $\Delta u = \chi_{\{u>0\}}, u \geq 0$ in $\R^N$
	such that the coincidence set $\{ u=0\}$ has non-empty interior
	and let
	\begin{align}
	p(x) := \lim \limits_{\rho \to \infty} \frac{u(\rho x)}{\rho^2}\quad \text{ in } L^\infty(\partial B_1) \quad \text{and} \quad \cN(p) := \{p=0\}.
	\end{align}
	Setting $n:= \dim \cN(p)$, then
	\\
	the coincidence set is an ellipsoid,
	a paraboloid or a cylinder with an
	ellipsoid or a paraboloid as base. \\ Furthermore the cross sections of the  paraboloid are given as the (up to scaling and translation) \emph{unique} ellipsoids that are given as the coincidence set $\{v'=0\}$ of  a solution of the ($N-n$)-dimensional obstacle problem with \emph{the same blow-down}, i.e. $v': \R^{N-n} \to [0,\infty)$ solves
	\begin{align}
	\Delta v' &= \chi_{\set {v'>0}} ~&&\text{ in  } \R^{N-n} \quad \text{ and } \\
	\lim \limits_{\rho \to \infty} \frac{v'(\rho x')}{\rho^2} &= p(x')=\lim \limits_{\rho \to \infty} \frac{u(\rho x)}{\rho^2}~&&\text{ in } L^\infty(\partial B_1) .
	\end{align}	
\end{prop}

\begin{defn}[Coincidence set]\label{def:coincidence_set}
	\mbox{} \\
	For solutions $u$ of the obstacle problem \eqref{eq:obstacle-local}, we define the \emph{coincidence set} $\cC$ to be
	\begin{align}
	\cC := \set {u=0}.
	\end{align}
\end{defn}

\begin{rem} 	It is known that the coincidence  $\cC$ of a \emph{global} solution of the obstacle problem with constant coefficients is \emph{convex} (see e.g. \cite[Theorem 5.1]{PetrosyanShahgholianUraltseva_book}).
\end{rem}

\begin{defn}[Ellipsoids and Paraboloids] \label{def:ellipsoid_and_paraboloid}
	\mbox{} \\
	We call a set $E \subset \R^N$ \emph{ellipsoid} if after translation and rotation
	\begin{align}
	E = \set {x \in \R^N : \sum \limits_{j = 1}^N \frac{x_j^2}{a_j^2} \leq 1   }
	\end{align}
	for some $a \in (0,\infty)^N$.  We call a set $P \subset \R^N$ a \emph{paraboloid}, 
	if after translation and rotation  $P $ can be represented as 
	$$
	P = \set { (x',x_N) \in \R^N : x' \in \sqrt{x_N} E'   } ,
	$$	  
	where  $E' $  is an $(N-1)$-dimensional ellipsoid.
\end{defn}

\begin{defn}[Local solutions and singular points] \label{solandcod}
	\mbox{ }\\
	\begin{enumerate}[(i)] 
		\item We define a local solution of the obstacle problem (with Lipschitz-coefficients), to be a non-negative function $u$ solving 
		\begin{equation} 
		\Delta  u = \coeff(x)  \chi_{\set {u>0}} \quad \text{ in } \Omega \qquad , \qquad \coeff \in \operatorname{Lip}(\Omega, [\coeff_0, \infty)), \coeff_0 >0. 
		\end{equation}

		\item  \label{item:definition_order_1_singularity} For a local solution $u$ of the obstacle problem (with Lipschitz-coefficients), we call  
		$x^0 \in \partial \{u>0\} \cap B_1$ an order-$n$ singular point\footnote{The order here refers to the dimension of the zero set of the blow-up polynomial.}
		if 
		\begin{align} \label{eq:type_of_singularity_in_singularities_theorem}
		\frac{u(rx+x^0)}{r^2} \to p(x) \quad \text{ in } C_{\operatorname{loc}}^{1,\alpha}(\R^N) \text{ as } r \to 0
		\end{align}  
		where
		\begin{align}\label{pcond}
		p(x) &= x^T A x ~ \text{ for some  } A \in \R^{N \times N} \text{ with } \operatorname{dim} \operatorname{ker}(A) =n \text{ and } \\
		p(x) &\geq c_p \abs{x}^2 ~ \text{ for some } c_{p}>0 \text{ and all } x \in \operatorname{ker}(A)^\perp. 
		\end{align} 
		We define $\cN(p) := \operatorname{ker}(A)$ and $\Sigma_n$ to be the set of all order-$n$ singular points.  
	\end{enumerate}
\end{defn}

\begin{rem}
Note that \emph{any} singular point of a solution of the obstacle problem (with Lipschitz-coefficients) is an order-$n$ singular point for some $n \in \N$.
\end{rem}

\section*{Acknowledgment}
The authors thank the anonymous referee for his valuable suggestions on the improvement of the presentation of this work.

\section{Main part}


In the following lemma we show that \emph{every} limit of a blow-up \emph{with moving center} towards a singularity $x^0$ has the \emph{same blow-down, i.e. the blow-up of $u$ at $x^0$}.

\begin{lem}\label{blow-down-blow-up}
	Let $(x^k)_{k \in \N} \subset \partial \set {u>0}$ be a sequence of
	free boundary points approaching a singular free boundary point $x^0 \in \Sigma_n$ (cf. Definition \ref{solandcod} \eqref{item:definition_order_1_singularity}).
	
	Moreover let $u_0$ be a blow-up limit of $u$, i.e.
	suppose that $x^k \to x^0$ and $r_k\to 0$ as $k\to\infty$
	and that
	$$ \frac{u(x^k+r_k \cdot)}{{r_k}^2}\to u_0 \quad \text{ in } C^{1,\alpha}_{\operatorname{loc}}(\R^N).$$
	
	Then either $u_0$ is a half-space solution, or the unique blow-down limit 
	\begin{equation}\label{blow-down-v}
	v(x) : = \lim_{\rho \to +\infty}\frac{u_0(\rho x)}{\rho^2}
	\end{equation}
	is the
	polynomial $p$ of   \eqref{eq:type_of_singularity_in_singularities_theorem}, i.e. the blow-up of $u$ at $x^0$. 
	\\
	If we assume in addition that $|\{ u_0=0\}|=0$, then
	$u_0= p $.
\end{lem}

\begin{proof}
	We may assume that the coefficients $\coeff$ are normalized such that $\coeff(x^0)=1$. It is known (cf. \cite[proof of Proposition 3.17 (iii)]{PetrosyanShahgholianUraltseva_book}) that $u_0$ is a global solution of the obstacle problem, i.e.
	\begin{align}\label{eq:obstacle_problem_solved_by_blow-up}
	\Delta u_0 = c(x^0) \chi_{\set {u_0>0}  } = \chi_{\set {u_0>0}  } \quad , \quad u_0 \geq 0 \quad \text{ in } \R^N.
	\end{align}
	It is furthermore known (see  proof of Theorem II, Case 2, and 3 in \cite{CaffarelliKarpShahgolian_Annals_2000}) 
	that every blow-down limit of any global solution of the obstacle problem \eqref{eq:obstacle_problem_solved_by_blow-up} is either a  half-space solution or a homogeneous polynomial $q$ of degree $2$ satisfying $\Delta q \equiv 1$. 
	Moreover it is known (see  proof of Theorem II, Case 2 in \cite{CaffarelliKarpShahgolian_Annals_2000}) that if  the blow-down of any solution is a half-space solution then 
	the solution itself has to be a half-space solution.
	We thus  conclude that either
	$u_0$ is a half-space solution, or the blow-down $v$ in \eqref{blow-down-v}
	is a homogeneous polynomial  $q$ of degree $2$ satisfying $\Delta q \equiv 1$.

	Define now  $\phi(h, r, y)$ to be the ACF-functional 
	\begin{align}
	\phi(h, r, y) := \frac{1}{r^4} \int \limits_{B_r(y)} \frac{\abs {\nabla h^+}^2}{\abs {x}^{N-2} }  \dx{x} \int \limits_{B_r(y)} \frac{\abs {\nabla h^-}^2}{\abs {x}^{N-2} }  \dx{x},
	\end{align}
	which is `almost non-decreasing' in $r$, see \cite[Theorem 1.6]{CaffarelliJerisonKenig_new_monotonicity_formula_2002}.
	We infer from \cite[Theorem 1.6]{CaffarelliJerisonKenig_new_monotonicity_formula_2002} that there is $C< \infty$ such that for each $e \in \partial B_1$, $\rho >0$ and $\epsilon >0$, choosing first $\tilde{r}_0$ sufficiently small and then $k$ sufficiently large,
	\begin{align}
	\phi(\partial_e u_0, \rho, 0) &{\leq} \epsilon + \phi \bra { \partial_e \frac{u(r_k x +x^k)}{r_k^2}, \rho,0   } 
	= \epsilon + \phi \bra { \partial_e u , r_k \rho, x^k    }\\
	&\leq  \epsilon + (1+ \tilde{r}_0) \phi \bra {\partial_e u , \tilde{r}_0, x^k} + C \tilde{r}_0
	\leq 2 \epsilon + (1+ \tilde{r}_0)\phi \bra { \partial_e u, \tilde{r}_0, x^0  } + C \tilde{r}_0 \\ &\leq 3 \epsilon + \phi \bra {\partial_e p,1,0}.
	\end{align}
	Hence for every $\rho >0$ 
	\begin{align}\label{spec1}
	\phi \bra {\partial_e u_0, \rho, 0} \leq \phi \bra {\partial_e p ,1,0},
	\end{align}
	and passing to the limit $\rho \to \infty$ we get that
	\begin{align}
	\phi \bra {\partial_e p , 1,0   } \geq \phi \bra {\partial_e u_0, \rho,0} = \phi \bra {\partial_e \frac{u_0(\rho \cdot)}{\rho^2},1,0} \to \phi \bra {\partial_e v,1,0 }
	\end{align} as $\rho \to \infty$.
	We obtain that for all $e \in \partial B_1$
	\begin{align} \label{eq:ACF_comparison_between_blowup_and_blowdown}
	\phi \bra {\partial_e p , 1,0   } \geq \phi \bra {\partial_e q,1,0 }.
	\end{align} 
	Let us now express $p$ and $q$ as
	\begin{align}
	p(x) = x^T A x \quad \text{ and } \quad q(x) = x^T Q x,
	\end{align}
	where $A$ and $Q \in \R^{N \times N}$ are symmetric positive semidefinite such that $\operatorname{tr}(A)= \operatorname{tr}(Q) = \frac{1}{2}$.
	From \eqref{eq:ACF_comparison_between_blowup_and_blowdown} we conclude that for every $e \in \partial B_1$ 
	\begin{align} \label{eq:comparison_of_matrices_in_directions}
	\abs{Q e}^2 \leq \abs{A e}^2.
	\end{align}
	Using \cite[Lemma 14]{Caffarelli-revisited} we obtain that $A=Q$. Hence $q \equiv p$. \\
	In the special case $|\{ u_0=0\}|=0$,
	we infer from the equation $\Delta u_0\equiv 1$ and the quadratic growth of $u_0$
	---using Liouville's theorem--- that it is a quadratic polynomial.
	The fact that $u_0(0)=|\nabla u_0(0)|=0$ and the asymptotics of $u_0$ imply
	that $u_0\equiv p$.
\end{proof}

\begin{prop} \label{lem:existence_of_non_trivial_scaling}
	Let $n \geq 1$ and let $u$ be a solution of \eqref{eq:obstacle-local} and let $(x^k)_{k \in \N} \subset \partial \set {u>0}$ be a sequence of \emph{regular} free boundary points approaching a singular free boundary point $x^0 \in \Sigma_n$.
	Then there is a sequence of rescalings $(r_k)_{k \in \N} \subset (0,\infty)$, $r_k \to 0$ as $k \to \infty$ such that 
	\begin{align}
	u_k(x) := \frac{u(r_k x + x^k)}{r_k^2} \to u_0 \quad \text{ in } C^{1,\alpha}_{\operatorname{loc}}(\R^N) \text{ as } k \to \infty,
	\end{align}
	where $\set {u_0=0}$ is \emph{either a cylinder {over} an ($N-n+1$)-dimensional paraboloid or a cylinder with an ($N-n$)-dimensional ellipsoid as base, and the cylindrical directions are in $\cN(p)$}.
\end{prop}

\begin{proof}
	We may assume that the coefficients $\coeff$ are normalized such that $\coeff(x^0)= 1$.\\ 
	\textbf{Step 1.} \emph{Construction of the rescaling $r_k$.}\\
	For any $\epsilon >0$ we obtain from Definition \ref{solandcod} \eqref{item:definition_order_1_singularity} that there is $r_0(\epsilon) >0$ such that for all $0<r< r_0(\epsilon)$
	\begin{align}
	\abs { \frac{u(r x+x^0)}{r^2} -p(x)      } < \epsilon \quad \text{ for all } x \in B_2.
	\end{align}
	Let us from now on confine ourselves to $k>K(\epsilon)$, where $K(\epsilon)$ is such that
	\begin{align}
	\abs{x^k-x^0} < r_0(\epsilon) \quad \text{ for all } k >K(\epsilon).
	\end{align}
	Let us furthermore for each $k>K(\epsilon)$ choose $\overline{r}_k >0$ such that
	\begin{align}
	\abs{x^k-x^0}< \overline{r}_k < \min \set {2 \abs{x^k-x^0},  r_0(\epsilon) }
	\end{align}
	implying that for all $k > K(\epsilon)$
	\begin{align}
	\abs {\frac{x^k-x^0}{\overline{r}_k}} <1.
	\end{align}
	This choice of $\overline{r}_k$ implies that for
	\begin{align}
	u_{\overline{r}_k,x^k}(x):= \frac{u \bra {\overline{r}_k x +x^k}}{\overline{r}_k^2} 
	\end{align}
	it holds that
	\begin{align}
	\abs{   \set {u_{\overline{r}_k,x^k} =0} \cap B_1} \leq \abs{ B_{\sqrt{\frac{\epsilon}{c_p}}} (\cN(p))   \cap B_1}. 
	\end{align}
	Now using that $x^k$ is a \emph{regular} free boundary point there is a half-space solution $H^k(x)$ and a scaling $0< \underline{r}_k < \overline{r}_k$ such that
	\begin{align}
	\abs { \frac{u(\underline{r}_k x + x^k)}{\underline{r}_k^2} -H^k(x)      } < \epsilon \quad \text{ for all  } x \in B_1.
	\end{align}
	The non-degeneracy Lemma (cf. \cite[proof of Lemma 3.1]{PetrosyanShahgholianUraltseva_book}) implies that
	\begin{align}
	u_{\underline{r}_k,x^k}(x):= \frac{u(\underline{r}_k x + x^k)}{\underline{r}_k^2}  = 0 \quad \text{ in } B_1 \setminus \set {B_{\sqrt{\frac{2N\epsilon}{c_0}}}\bra {\set {H^k>0}}}.
	\end{align}
	So for $\epsilon>0$ small enough,
	\begin{align}
	\abs { \set {u_{\underline{r}_k,x^k} =0} \cap B_1    } > \frac{1}{4} \abs{B_1} > \abs { \set {u_{\overline{r}_k,x^k}=0} \cap B_1    }. 
	\end{align}
	Setting
	\begin{align}
	u_{r, x^k}(x):= \frac{u(r x + x^k)}{r^2},
	\end{align}
	we conclude that $\abs { \set {u_{r,x^k} =0} \cap B_1    } $ is continuous {in $r$}:
	\begin{align} \label{eq:local_continuity_of_coincidence_set_in_measure}
	\begin{split}
	\abs { \set {u_{r,x^k} =0} \cap B_1    } &= \abs {B_1} - \int \limits_{B_1} \chi_{\set {u_{r,x^k} >0}} = \abs {B_1} - \int \limits_{B_1} \frac{  \Delta u_{r,x^k}(x) }{c(x^k + rx)}  \dx{x}   \\
	&\to  \abs {B_1} - \int \limits_{B_1} \frac{  \Delta u_{\tilde{r},x^k}(x) }{c(x^k + \tilde{r}x)}  \dx{x}   = \abs {B_1} - \int \limits_{B_1} \chi_{\set {u_{\tilde{r},x^k} >0}} \\ &= \abs { \set {u_{\tilde{r},x^k} =0} \cap B_1    } \quad \text{ as } r \to \tilde{r}.
	\end{split}
	\end{align}
	Note that here we used the strong $W^{2,p}_{\loc}$-convergence (cf. \cite[proof of Proposition 3.17 (v)]{PetrosyanShahgholianUraltseva_book}).
	Consequently,
	there is $ r_k \in \bra{\underline{r}_k , \overline{r}_k}$ such that
	\begin{align}\label{eq:measure_of_coincidence_set_is_that_of_quater_ball}
	\abs { \set {u_{r_k,x^k} =0} \cap B_1    } = \frac{1}{4} \abs{B_1}.
	\end{align}

	\noindent
	\textbf{Step 2.} \emph{Identifying the limit solution.}\\
	Let us now set $u_k(x):= \frac{u \bra {{r}_k x +x^k}}{{r}_k^2} $.
	It is known that (passing if necessary to a subsequence)
	\begin{align}
	u_k \to u_0 \quad \text{ in } C^{1,\alpha}_{\operatorname{loc}}(\R^N) \text{ as } k \to \infty,
	\end{align}
	where $u_0$ is  a global solution of the obstacle problem 
	\begin{align}
	\Delta u_0 = c(x^0) \chi_{\set {u_0>0}  } = \chi_{\set {u_0>0}  } \quad , \quad u_0 \geq 0 \quad \text{ in } \R^N
	\end{align}
	(cf.  \cite[Proof of Proposition 3.17 (iii)]{PetrosyanShahgholianUraltseva_book}).
	Employing once more \eqref{eq:local_continuity_of_coincidence_set_in_measure} together with \eqref{eq:measure_of_coincidence_set_is_that_of_quater_ball} and the strong $W^{2,p}_{\loc}$-convergence (cf. \cite[proof of Proposition 3.17 (v)]{PetrosyanShahgholianUraltseva_book}) we find that
	\begin{align} \label{eq:measure_of_coicidence_set_of_blow_up_with_moving_centre_in_ball}
	\tfrac{1}{4} \abs {B_1} = \abs {\set {u_k =0} \cap B_1} \to  \abs {\set {u_0 =0} \cap B_1} \quad \text{ as } k \to \infty.
	\end{align}
	This implies that $u_0$ cannot be a half-space solution. By Lemma \ref{blow-down-blow-up}, the unique blow-down limit of $u_0$
	is the polynomial $p$ and from \eqref{eq:measure_of_coicidence_set_of_blow_up_with_moving_centre_in_ball} we infer that $|\{ u_0=0  \}| \neq 0$. Together with the fact that $\{u_0=0\}$ is convex (cf. \cite[Theorem 5.1]{PetrosyanShahgholianUraltseva_book}) and the fact that convex sets with empty interior have zero Lebesgue-measure this implies that $\operatorname{int}(\{u_0=0\} ) \neq \emptyset$. From the ACF-argument in 
	\cite[proof of case 2 of Theorem II]{CaffarelliKarpShahgolian_Annals_2000} we conclude that $u_0$ is monotone in all directions in $\cN(p)$ and since $n = \operatorname{dim}(\cN(p)) \geq 1$ this implies that $\{u_0=0\}$ is \emph{unbounded} (in all directions in $\cN(p)$). 
	Furthermore the fact that $p$ is non-degenerate in all directions in $\cN(p)^\perp$	implies that $\{ u_0=0\}$ must be bounded in all these directions.\\
	Now we are able to apply Proposition \ref{thm:MainTheorem_Intro_I} to conclude that $\{u_0=0\}$ is either a cylinder {over} an $(N-k)$-dimensional paraboloid or a cylinder {over} an $(N-k)$-dimensional ellipsoid. Since the blow-down of $u_0$ is $p$ and vanishes in precisely $n$ independent directions we conclude that $\{u_0=0\}$ is either a cylinder with an ($N-n+1$)-dimensional paraboloid or an ($N-n$)-dimensional ellipsoid as base. 
\end{proof}

\begin{proof}[Proof of the Main Theorem]
	\mbox{} \\
	In the following we denote by \emph{cross section }at a free boundary point $x \in \partial \{u>0\}$ the set $\{y \in B_\delta(x) : u(y)=0~,~ y'' =x''\}$.
	\\
	\textbf{\eqref{item:first_part_of_theorem_singularities}  {\emph {Cross sections are $C^2$-perturbations of ellipsoids.}}}\\
	{\bf Step 1.} \emph{Cross sections that contain at least one \emph{regular} free boundary point.}
	\\
	Suppose towards a contradiction that 
	$(x^k)_{k \in \N} \subset \partial \set {u>0}$ is a sequence of \emph{regular} free boundary points such that
	\begin{align}
	x^k \to x^0 \quad \text{ as } k \to \infty  , \quad d_k := d((x^k)'') >0 \text{ for all } k \in \N
	\end{align}
	and that the statement in (i) does not hold.
	Passing to a subsequence,
	\begin{align}
	\tilde u_k(x) = u_{d_k,x^k}(x) := \frac{u(x^k+d_k x)}{d_k^2} \to \tilde u_0(x) \quad \text{ in } C_{\operatorname{loc}}^{1,\alpha}(\R^N) \text{ as } k \to \infty.
	\end{align}
	Note that because $x^0$ is a singular point it holds that $d((x^k)'') \to 0$ as $x^k \to x^0$.
	From Proposition \ref{lem:existence_of_non_trivial_scaling} we know that there is another subsequence and scalings $r_k \to 0$ as $k \to \infty$ such that
	\begin{align}
	u_k(x) = u_{r_k,x^k}(x) := \frac{u(x^k+r_k x)}{r_k^2} \to u_0(x) \quad \text{ in } C_{\operatorname{loc}}^{1,\alpha}(\R^N) \text{ as } k \to \infty
	\end{align}
	and $\set {u_0=0}$ is either a paraboloid (only possible if $n=1$) or a cylinder with an ($N-n+1$)-dimensional paraboloid or an ($N-n$)-dimensional ellipsoid as base. We  distinguish three cases:
	\begin{enumerate}[1.]
		\item \label{item:item3_in_proof_of_isolation_lemma} For a subsequence, $\frac{d_k}{r_k} \to \lambda \in (0,\infty)$ as $k \to \infty$.
		\item \label{item:item2_in_proof_of_isolation_lemma} For a subsequence, $\frac{d_k}{r_k} \to \infty$ as $k \to \infty$.
		\item\label{item:item1_in_proof_of_isolation_lemma} For a subsequence, $\frac{d_k}{r_k} \to 0$ as $k \to \infty$.
	\end{enumerate}
	
	As part of our proof works in the affine subspace $\{ y\in \R^N:  y'' =(x^k)''  \}$, let us use the notation 
	$$u_k'(x') := u_k(x',0), \qquad u_0'(x') := u_{0}(x',0), $$
	$$  \tilde u_k'(x') := \tilde u_k(x',0), \qquad \tilde u_0'(x') := u_{0}(x',0).$$
	
	Now in  Case 1 we have that $\tilde{u}_0$ is a rescaled version of $u_0$ and Lemma \ref{blow-down-blow-up} tells us
	that 
	the blow-down limit of $\tilde u_0$ is the polynomial $p$
	whence 
	Proposition \ref{thm:MainTheorem_Intro_I} implies that a scaled and translated instance of
	$\{ \tilde u_0'=0\}$ is the ellipsoid $E'$ and $\{ \tilde{u}_0 =0 \}$ is cylindrical in $n$-directions (ellipsoid-cylinder case) or ($n-1$)-directions (paraboloid-cylinder case) that are contained in $\cN(p)$.
	
	Finally, using 
	stability of regular  free boundaries\footnote{\label{foot-flatness}
	Here we use the fact that $C^0$-closeness of $\tilde{u}_0$ and $\tilde{u}_k$ together with nondegeneracy (see \cite[Lemma 3.1]{PetrosyanShahgholianUraltseva_book}) implies that $\partial \{u_k >0\}$ is flat and by \cite{Caff-81-compactness} uniformly bounded in $C^{2,\beta}_{\operatorname{loc}}$. Hence by compactness we get convergence for a subsequence in $C^{2,\alpha}_{\operatorname{loc}}$.
	}
	we obtain
	$C^2$-convergence of a subsequence of sets $\{ \tilde{u}_k'=0\}$ to the ellipsoid $E'$, and of $\{ \tilde{u}_k =0 \}$ to a cylinder in $(N-n)$ or $(N-n+1)$ independent directions that are contained in $\cN(p)$, a contradiction to our assumption that (i) does not hold.
	\\
	Case 2 is more involved as we have to exclude the possibility of tiny components of the coincidence set, with cross sections being relatively far from each other.
	First, we  show that $\tilde u_0$ is not a half-space  solution. Indeed,
	using the ACF monotonicity formula (in the same notation as in the proof of Proposition \ref{lem:existence_of_non_trivial_scaling}) we conclude from \cite[Theorem 1.6]{CaffarelliJerisonKenig_new_monotonicity_formula_2002} that there is $C< \infty$ such that for every $e \in \partial B_1 \setminus  \cN(p)$,
	\begin{align}
	0< \phi \bra {\partial_e u_0,1,0} & \gets \phi \bra {\partial_e u_k,1,0} = \phi \bra {\partial_e u_{r_k, {x}^k},1,0} 
	\leq (1+d_k)\phi \bra {\partial_e u_{d_k, {x}^k},1,0} + {C} d_k \\
	&= (1+d_k) \phi \bra {\partial_e \tilde{u}_k,1,0} + {C} d_k
	\to \phi \bra {\partial_e \tilde{u}_0,1,0} 
	\end{align}
	as $k \to \infty$. Since the ACF-monotonicity functional is zero for half-space solutions we conclude that   $\tilde{u}_0$ is not a half-space  solution.
	
	Using the fact that $d((x^k)'') \to 0$ as $x^k \to x^0$ we can again invoke  Lemma \ref{blow-down-blow-up}, which  implies that 
	the unique blow-down limit of $\tilde u_0$ is the polynomial $p$.
	This in turn enables us to apply Proposition \ref{thm:MainTheorem_Intro_I}.
	We thus obtain that either $|\{ \tilde u_0=0\}|=0$, or $\{\tilde u_0 =0 \}$ is a 
	paraboloid (only possible if $n=1$) or a cylinder with an $(N-n+1)$-dimensional paraboloid or an $(N-n)$-dimensional ellipsoid as base.
	If $|\{\tilde u_0 =0 \}| =0$, Lemma \ref{blow-down-blow-up} tells us that
	$\tilde u_0 \equiv p$ and since $p(x) \geq c_p |x'|^2$ this is a contradiction to the definition of $d_k$ by which
	$\partial B_1' \cap \{ \tilde u_0'=0\} \ne \emptyset$.
	In case that $\{ \tilde{u}_0 =0\}$ is a paraboloid or a cylinder with a paraboloid or an ellipsiod as base, using once more the information that 
	the blow-down limit of $\tilde u_0$ is the polynomial $p$,
	Proposition \ref{thm:MainTheorem_Intro_I} implies that a scaled and translated instance of
	$\{ \tilde u_0'=0\}$ is the ellipsoid $E'$.
	Together with local $C^2$-convergence of the sets
	$\{ \tilde u_k=0\}$, this poses a contradiction in Case 2.

	In Case 3 we first recall that $\set {u_0=0}$ is either a paraboloid (only possible if $n=1$) or a cylinder with paraboloidal or ellipsoidal base, but the 
	assumption $d_k/r_k \to 0$ as $k \to \infty$ forces  
	$\set {u_0=0}$ to be  a paraboloid or cylinder {over} a paraboloid with tip at the origin.
	Assume towards a contradiction that this observation is not true, i.e. that $\set {u_0=0}$ is a paraboloid / cylinder {over} a paraboloid with tip not in the origin or a cylinder with an ellipsoid as base. This implies that $\set {u_k'=0}$ has positive diameter that is (uniformly in $k$) bounded from below. But this contradicts the assumption that $\frac{d_k}{r_k} \to 0$ as $k \to \infty$.
	So $\partial \{ u_0=0\}$ is given by the graph of a quadratic polynomial
	$f_0(x')$ satisfying
	$f_0(x')\ge c_1 |x'|^2$. There is ${e}'' \in \partial B_1''$ and $\eta >0$ such that $ \{u_0 =0\} \cap B_\eta = \{ y \in B_\eta : {e}'' \cdot y'' \geq f_0(y')    \}$
	and $e''$ my be expressed as $$e'' =  \frac{  \int_{ \{u_0=0\} \cap B_{1}  } (-y)'' \dx{y} }{ \abs{\int_{ \{u_0=0\} \cap B_{1}  } (-y)'' \dx{y}}  } $$ and $\nu''(x^k) \to e''$ as $x^k \to x^0$.
	{S}tability of regular solutions (relying on flatness-implies-regularity, see Footnote \ref{foot-flatness})
	implies that
	$\partial \{ u_k=0\}$ is for sufficiently large $k$ given by the graph
	of a $C^{2,\alpha}(\R^{N-1})$-function $f_k$ such that
	\begin{align}
	\{ u_k = 0 \} \cap B_\eta = \set { y \in B_\eta : {e}'' \cdot y'' \geq f_k(y', \pi_{{e}''} y'')   },
	\end{align}
	(where $\pi_{{e}''} := I_{{e}''} \circ P_{{e}''} : \R^n \to \R^{n-1}$ and $P_{{e}''} : \R^n \to \{ y'' \in \R^n : {e}'' \cdot y'' =0  \} \subset \R^n$ is the orthogonal projection in the direction ${e}''$ and $I_{{e}''} : \{  y'' \in \R^n : {e}'' \cdot y'' =0  \} \hookrightarrow \R^{n-1}$ the canonical isomorphism (given by rotation)) and
	$f_k\to f_0$ in $C^2$ locally
	in $\R^{N-1}$, the tip of the graph of $f_k(y', \pi_{{e}''} 0)$ converges to
	the origin as $k\to\infty$.
	Translating each graph, we may assume that
	$0=f_k(0)=|\nabla f_k(0)|$ for all $k \in \N$.
	Finally, we introduce the rescaled functions\footnote{This is  an inhomogeneous scaling of the original free boundary, but a homogeneous scaling of each cross section.} 
	\begin{align}
	g_k(y',\pi_{{e}''}y'') :={\bra {\frac{d_k}{r_k} }^{-2 }} {f_k \bra {\frac{d_k}{r_k}  \bra {y' , \pi_{{e}''}y''  }    } }.
	\end{align}
	
	Then $g_k$ converges in $C^2$ locally in $\R^{N-1}$ to the same polynomial
	$f_0$. The set $\{ \tilde{u}_k=0\} \cap \{y \in B_\eta : y'' \cdot {e}''=0\}$ corresponds for large $k$ to
	$\{ y \in B_\eta: y'' \cdot {e}'' = t_k \geq g_k(y' , \pi_{{e}''} y''  )\}$  for some $t_k$. Now
	the diameter of $\{ u_k'=0\}$
	being $d_k/r_k$ implies the diameter of $\{ y': g_k(y', \pi_{{e}''} 0  )\le t_k\}$ being $1$
	such that $0< T_1 < t_k < T_2 <+\infty$ for all sufficiently large $k$.
	Finally note that
	by the implicit function theorem, the sublevel sets 
	$\{ g_k \leq t_k \}$ converge for every sequence
	$(t_k)_{k\in \N} \subset [T_1,T_2]$ locally
	in $C^2$ to a scaled instance of
	$E'\times \R^{n-1}$ (here we used Proposition \ref{thm:MainTheorem_Intro_I} as in Case 1 and Case 2). So we obtain a contradiction in Case 3.
	\\
	{\bf Step 2.} \emph{Cross sections that do not contain a \emph{regular} free boundary point} \\
	Last, suppose that there is a sequence $x^k\to x^0$ as $k\to\infty$
	such that $d((x^k)'')>0$ but the set
	$\{ y\in B_\delta(x^0): y''=(x^k)'' \}$ contains no
	regular free boundary point.
	Then that set contains at least two singular free boundary points
	$x^k$ and $\tilde x^k$.
	Let the homogeneous quadratic polynomial $q_k$ be 
	a blow-up limit of $u$ at $x^k$.
	Setting $r_k := |\tilde x^k-x^k| \to 0$ as $k \to \infty$ and
	$u_k(x) := u(x^k+r_kx)/r_k^2$, passing if necessary to a subsequence,
	we may assume that
	$u_k \to u_0$ in $C^{1,\alpha}_{\operatorname{loc}}(\R^N)$.
	Using the ACF monotonicity formula
	we may estimate for each $e \in \partial B_1$, $\rho >0$ and $\epsilon >0$, choosing first $\tilde{r}_0$ sufficiently small and then $k$ sufficiently large,
	\begin{align}
	&\phi \bra {\partial_e q_k,1,0}  = \lim \limits_{r\to 0} \phi(\partial_e u, r, x^k)  \leq
	\phi \bra { \partial_e u , r_k \rho, x^k    } =
	\phi \bra { \partial_e u_k , \rho, 0    } 
	\\
	&\leq \epsilon + \phi \bra {\partial_e u , \tilde{r}_0, x^k} 
	\leq 2 \epsilon + \phi \bra { \partial_e u, \tilde{r}_0, x^0  } {\leq} 3 \epsilon + \phi \bra {\partial_e p,1,0}.
	\end{align}
	Passing if necessary to another subsequence we may assume that
	$q_k\to q$ as $k\to\infty$.
	We obtain that for all $e \in \partial B_1$
	\begin{align}
	\phi \bra {\partial_e q , 1,0   } 
	\leq \phi \bra { \partial_e u_0 , \rho, 0 }
	\leq \phi \bra {\partial_e p,1,0 },
	\end{align} 
	whence \cite[Lemma 14]{Caffarelli-revisited} implies
	$q\equiv p$. But then the ACF monotonicity formula (cf. \cite[Theorem 2.9]{PetrosyanShahgholianUraltseva_book}) implies
	that $u_0$ itself is a homogeneous quadratic polynomial
	which must by \cite[Lemma 14]{Caffarelli-revisited} equal
	$p$.
	Thus $u_0>0$ in $\{ y: y''=0\}\setminus \{ 0\}$ contradicting
	the assumption 
	$u(\tilde x^k)=0$  and the choice of the scaling $r_k$.
	\\
	\textbf{\eqref{item:second_part_of_theorem_singularities}  \emph{ Behavior close to diameter zero points.}}\\
	Here we will prove \eqref{item:second_part_of_theorem_singularities} of Theorem \ref{thm:MainTheoremII}, so let $x$ be a free boundary point
	close to $x^0$ such that $d(x'' )=0$.
	We will distinguish two cases:
	\\ 
	\emph{ Case 1: $x$ is a singular free boundary point.} \\
	Let us assume towards a contradiction that there is a sequence of \emph{singular} points $(x^k)_{k \in \N}$ satisfying $x^k \to x^0$ as $k \to \infty$ and let $q_k$ be the blow-up polynomial of $u$ in $x^k$, i.e.
	\begin{align}
	u_{r,x^k} (z):= \frac{u(rz + x^k)}{r^2} \to q_k(z)  \quad \text{ in } C_{\loc}^{1,\alpha}(\R^N)  \text{ as } r \to 0.
	\end{align}
	Let $P, Q_k \in \R^{N \times N}$ satisfying $\operatorname{tr}(P) = \operatorname{tr}(Q_k) = \frac{1}{2}$ be symmetric, positive semidefinite matrices such that $p(z) = z^T P z$ and $q_k(z) = z^T Q_k z$ for all $z \in \R^N$. From \cite[Corollary 10 d)]{Caffarelli-revisited} we infer that $|P-Q_k| \to 0$ as $k \to \infty$. It follows that for every $\mu >0$: $\{q_k =0\} \cap B_1 \subset B_\mu(\{p=0\}) \cap B_1$ for sufficiently large $k$. \\
	Combining this with the known fact that that the coincidence set close to singular points is contained in an arbitrarily small cone around the zero set of the blow-up polynomial, i.e. for each singular point $x^k$ and each $\epsilon >0$ there is $\kappa(\epsilon, x^k)>0$ such that
	\begin{align} \label{eq:inclusion_of_coincidence_set_in_cone}
	\{u=0\} \cap B_{\kappa(\epsilon, x^k)}(x^k) \subset \set {y \in B_{\kappa(\epsilon, x^k)}(x^k) : \epsilon |y-x^k| > \operatorname{dist}(y-x^k,   \{q_k =0\}) },	
	\end{align}
	concludes the proof.
	\\
	Inclusion \eqref{eq:inclusion_of_coincidence_set_in_cone} can be seen as follows. If this fact is not true, then for a  \emph{singular} free boundary point $x^k$ there is a sequence of points $(x^j)_{j \in \N} \subset \{u=0\}$ satisfying $x^j \to x^k$ as $j \to \infty$ such that $\epsilon |x^j-x^k| \leq \operatorname{dist}(x^j-x^k, \{q_k =0\})$ for all $j \in \N$.
	Passing if necessary to a subsequence, $\tfrac{x^j-x^k}{|x^j-x^k|} \to \xi^k \in \partial B_1$ as $j \to \infty$ and 
	\begin{align} \label{eq:xi_bounded_away_from_q=0}
	\epsilon \leq   \frac{\operatorname{dist}(x^j-x^k, \{ q_k=0 \} )}{|x^j-x^k|}    \to \operatorname{dist}(\xi^k, \{q_k =0\})  \quad \text{ as } j \to \infty.
	\end{align}
	But then, setting $r_j := |x^j-x^k|$,
	\begin{align}
	q_k(\xi^k) \leftarrow u_{r_j,x^k} \bra { \tfrac{x^j-x^k}{|x^j-x^k|} } := \frac{u \bra{\tfrac{x^j-x^k}{|x^j-x^k|} r_j + x^k  } }{r_j^2} = \frac{u(x^j)}{r_j^2}  =0 \quad \text{ as } j \to \infty.
	\end{align}
	But this is a contradiction to \eqref{eq:xi_bounded_away_from_q=0}, so \eqref{eq:inclusion_of_coincidence_set_in_cone} is proved.
	\\ 
	\emph{Case 2: $x$ is a regular free boundary point.}\\
	Supposing towards a contradiction
	that the statement does not hold in any neighborhood of $x^0$
	we obtain  a sequence $(x^k)_{k\in \N}$ of
	regular free boundary points satisfying $d((x^k)''  )=0$
	converging to $x^0$ and (cf. Proposition \ref{lem:existence_of_non_trivial_scaling}) a sequence
	$r_k \to 0$ such that 
	$u_k = u_{r_k,x^k}$ converges to a solution $u_0$, that is a paraboloid solution or a paraboloid-cylinder solution with tip in the origin. Assume towards a contradiction that this is not true and $\{ u_0=0  \}$ is either a paraboloid / paraboloid-cylinder with positive diameter in the subspace $\set {y''=0}$ or a cylinder with an ellipsoid as base. Then non-degeneracy of solutions of the obstacle problem \cite[Lemma 3.1]{PetrosyanShahgholianUraltseva_book} together with the fact that these coincidence sets are convex and have non-empty interior implies that there is $\kappa>0$ such that for all sufficiently large $k \in \N$, $\set {u_k=0} \cap   \set {y''=0} \supset B'_\kappa$. But this contradicts the assumption that $d((x^k)'')=0$.
	
	So as {Step 1 Case 3} $\{ u_0=0\}$ is given by the graph of a quadratic polynomial
	$f_0(x')$ satisfying
	$f_0(x')\ge c_1 |x'|^2$ in the sense that there is ${e''} \in \partial B_1''$ and $\eta >0$ such that $\{u_0 =0\} \cap B_\eta = \{ y \in B_\eta : {e''} \cdot y'' \geq f_0(y')    \}$.
	Stability of regular solutions (relying on flatness-implies-regularity, see Footnote \ref{foot-flatness})
	implies that
	$\partial \{ u_k=0\}$ is for sufficiently large $k$ given by the graph
	of a $C^{2,\alpha}(\R^{N-1})$-function $f_k$ such that
	\begin{align}
	\partial \{ u_k = 0 \} \cap B_\eta = \set { y \in B_\eta : {e''} \cdot y'' = f_k(y', \pi_{{e''}} y'')   },
	\end{align}
	(where $\pi_{{e''}} := I_{{e''}} \circ P_{{e''}} : \R^n \to \R^{n-1}$ and $P_{{e''}} : \R^n \to \{ y'' \in \R^n : {e''} \cdot y'' =0  \} \subset \R^n$ is the orthogonal projection in the direction ${e''}$ and $I_{{e''}} : \{  y'' \in \R^n : {e''} \cdot y'' =0  \} \hookrightarrow \R^{n-1}$ the canonical isomorphism (given by rotation)) and
	$f_k\to f_0$ in $C^2$ locally
	in $\R^{N-1}$ and the tip of the graph of $f_k$ converges to
	the origin as $k\to \infty$.

	Since $D^2 f_0$ is a (constant) positive definite matrix depending only on $x'$
	it follows that $c_3 |x'|^2 \le f_k(x', \pi_{{e''}}x '') \le C_4 |x'|^2$
	for $x \in B_\eta$, proving the estimate claimed in the statement
	for our specific subsequence and thus yielding a contradiction. \\
	\textbf{\emph{The oscillation of $\nu''$ close to $x^0$.}}\\
	Note that $\operatorname*{osc}_{ y \in B_{d(x'')}(x) \cap \{u=0 \}} \nu''(y) = 0$ if $d(x'')=0$. 
	When $d(x'')>0$ we make only claims on $\nu''$ in \eqref{item:first_part_of_theorem_singularities_b} and 
	the statement concerning the oscillation follows
	from the already proven fact that each limit $\{\tilde u_0 =0\}$ is 
	paraboloid (only possible if $n=1$) or a cylinder with an $(N-n+1)$-dimensional paraboloid as base.
	As a consequence of the stability of the regular free boundary (see Footnote \ref{foot-flatness}) we only need to justify that our expression for $\nu''(x)$ converges to the direction in which the limit-paraboloid $\{\tilde{u}_0 =0\}$ opens as $x \to x^0$.
	
	We know that the paraboloid must point in a direction where the blow-down $p$ of $\tilde{u}_0$ -- i.e. the blow-up of $u$ at $x^0$ -- is degenerate, hence it points in a direction in $\cN(p)$. It remains to localize the direction in which the paraboloid(-cylinder) $\{\tilde{u}_0=0\}$ opens within $\cN(p)$. Because of the symmetry of the paraboloid(-cylinder) with respect to all direction in $\cN(p)$ but the direction in which the paraboloid opens, this special direction may be expressed by the following `center of mass'
	$$z \mapsto  \frac{\int_{ \{\tilde u_0=0\} \cap B_1 }(z-y)'' \dx{y}}{\left | \int_{ \{\tilde u_0=0\} \cap B_1} (z-y)'' \dx{y} \right | },$$
	where $z \in \partial \{\tilde{u}_0 >0\}$.

	
\end{proof}

\bibliographystyle{abbrv}
\bibliography{references}

\begin{thebibliography}{10}

\bibitem{Caffarelli_regularity_free_boundary_higher_dimensions_Acta77}
L.~A. Caffarelli.
\newblock The regularity of free boundaries in higher dimensions.
\newblock {\em Acta Math.}, 139(3-4):155--184, 1977.

\bibitem{Caff-81-compactness}
L.~A. Caffarelli.
\newblock Compactness methods in free boundary problems.
\newblock {\em Comm. Partial Differential Equations}, 5(4):427--448, 1980.

\bibitem{Caffarelli-revisited}
L.~A. Caffarelli.
\newblock The obstacle problem revisited.
\newblock {\em J. Fourier Anal. Appl.}, 4(4-5):383--402, 1998.

\bibitem{CaffarelliJerisonKenig_new_monotonicity_formula_2002}
L.~A. Caffarelli, D.~Jerison, and C.~E. Kenig.
\newblock Some new monotonicity theorems with applications to free boundary
  problems.
\newblock {\em Ann. of Math. (2)}, 155(2):369--404, 2002.

\bibitem{CaffarelliKarpShahgolian_Annals_2000}
L.~A. Caffarelli, L.~Karp, and H.~Shahgholian.
\newblock Regularity of a free boundary with application to the {P}ompeiu
  problem.
\newblock {\em Ann. of Math. (2)}, 151(1):269--292, 2000.

\bibitem{EFW_Complete_classification}
S.~Eberle, A.~Figalli, and G.~S. Weiss.
\newblock Complete classification of global solutions to the obstacle problem.
\newblock {\em arXiv:2208.03108}, 2022.

\bibitem{EberleShahgholianWeiss_global_solutions_and_local_analysis_2020}
S.~Eberle, H.~Shahgholian, and G.~S. Weiss.
\newblock On global solutions of the obstacle problem.
\newblock {\em arXiv:2005.04915 accepted by Duke Mathematical Journal}, 2022.

\bibitem{Figalli-Ros-Oton-Serra-Generic}
A.~Figalli, X.~Ros-Oton, and J.~Serra.
\newblock Generic regularity of free boundaries for the obstacle problem.
\newblock {\em Publ. Math. Inst. Hautes \'{E}tudes Sci.}, 132:181--292, 2020.

\bibitem{Figalli-Serra-2020-Inventiones}
A.~Figalli and J.~Serra.
\newblock On the fine structure of the free boundary for the classical obstacle
  problem.
\newblock {\em Invent. Math.}, 215(1):311--366, 2019.

\bibitem{PetrosyanShahgholianUraltseva_book}
A.~Petrosyan, H.~Shahgholian, and N.~Uraltseva.
\newblock {\em Regularity of free boundaries in obstacle-type problems}, volume
  136 of {\em Graduate Studies in Mathematics}.
\newblock American Mathematical Society, Providence, RI, 2012.

\bibitem{SavinYu2020regularity}
O.~Savin and H.~Yu.
\newblock Regularity of the singular set in the fully nonlinear obstacle
  problem.
\newblock {\em To appear in Journal of the European Mathematical Society},
  arXiv:1905.02308, 2020.

\bibitem{Schaeffer_examples_of_singularities77}
D.~G. Schaeffer.
\newblock Some examples of singularities in a free boundary.
\newblock {\em Ann. Scuola Norm. Sup. Pisa Cl. Sci. (4)}, 4(1):133--144, 1977.

\end{thebibliography}
\end{document}